\documentclass[a4paper,
fontsize=11pt,%
oneside,%
numbers=enddot]{scrartcl}
\KOMAoptions{DIV=12}    
\usepackage[obeyspaces,hyphens,spaces]{url}
\usepackage[utf8]{inputenc}
\usepackage[T1]{fontenc}
\usepackage{textcomp}
\usepackage[fleqn]{amsmath}
\usepackage{amsthm}
\usepackage{amssymb}
\usepackage{amsfonts}
\usepackage{libertine}
\usepackage[libertine,
slantedGreek,
nosymbolsc,
nonewtxmathopt,
subscriptcorrection]{newtxmath}
\usepackage[scaled=0.95,varqu,varl]{inconsolata}
\usepackage[scr=boondox]{mathalpha}   
\usepackage{euscript}   
\usepackage{leftindex}
\allowdisplaybreaks[1]  
\numberwithin{equation}{section}    
\usepackage[svgnames,hyperref]{xcolor}
\definecolor{dblue}{HTML}{0455BF}
\definecolor{dgreen}{HTML}{02724A}
\definecolor{dgreen2}{HTML}{025951}
\definecolor{dred}{HTML}{D90404}
\definecolor{dviolet}{HTML}{42208C}
\definecolor{labelkey}{HTML}{025951}
\definecolor{refkey}{HTML}{025951}
\addtokomafont{section}{\centering}

\usepackage{enumitem}
\setlist{itemsep=-2.0pt}

\usepackage{upref}  
\usepackage{hyperref}
\hypersetup{colorlinks=true,
linktocpage=true,
linkcolor=dblue,
citecolor=dgreen,
urlcolor=dred,
pdfencoding=auto,
hypertexnames=false}
\makeatletter
\g@addto@macro\th@plain{
\thm@headfont{\bfseries\sffamily}
\thm@notefont{}}
\g@addto@macro\th@definition{
\thm@headfont{\bfseries\sffamily}
\thm@notefont{}}
\g@addto@macro\th@remark{
\thm@headfont{\bfseries\sffamily}
\thm@notefont{}}
\makeatother
\theoremstyle{plain}
\newtheorem{theorem}{Theorem}[section]
\newtheorem{proposition}[theorem]{Proposition}
\newtheorem{corollary}[theorem]{Corollary}
\newtheorem{lemma}[theorem]{Lemma}

\theoremstyle{definition}

\newtheorem{example}[theorem]{Example}

\theoremstyle{remark}
\newtheorem{remark}[theorem]{Remark}

\usepackage[extdef=true]{delimset}
\DeclareMathDelimiterSet{\scal}[2]{
\selectdelim[l]<{#1}
\mathpunct{}\selectdelim[p]|
{#2}\selectdelim[r]>}
\newcommand{\Scal}[2]{%
\left\langle #1 \,\middle|\, #2 \right\rangle}

\newcommand{\menge}[2]{\bigl\{{#1}\mid{#2}\bigr\}} 
\DeclareMathDelimiterSet{\Menge}[2]{\selectdelim[l]\{
{#1}\selectdelim[m]|{#2}\selectdelim[r]\}}

\makeatletter
\def\upintkern@{\mkern-7mu\mathchoice{\mkern-3.5mu}{}{}{}}
\def\upintdots@{\mathchoice{\mkern-4mu\@cdots\mkern-4mu}%
{{\cdotp}\mkern1.5mu{\cdotp}\mkern1.5mu{\cdotp}}%
{{\cdotp}\mkern1mu{\cdotp}\mkern1mu{\cdotp}}%
{{\cdotp}\mkern1mu{\cdotp}\mkern1mu{\cdotp}}}
\makeatother
\DeclareFontFamily{OMX}{mdbch}{}
\DeclareFontShape{OMX}{mdbch}{m}{n}{ <->s * [0.8]  mdbchr7v }{}
\DeclareFontShape{OMX}{mdbch}{b}{n}{ <->s * [0.8]  mdbchb7v }{}
\DeclareFontShape{OMX}{mdbch}{bx}{n}{<->ssub * mdbch/b/n}{}
\DeclareSymbolFont{uplargesymbols}{OMX}{mdbch}{m}{n}
\SetSymbolFont{uplargesymbols}{bold}{OMX}{mdbch}{b}{n}
\DeclareMathSymbol{\upintop}{\mathop}{uplargesymbols}{82}
\DeclareMathSymbol{\upointop}{\mathop}{uplargesymbols}{"48}
\makeatletter
\renewcommand{\int}{\DOTSI\upintop\ilimits@}
\renewcommand{\oint}{\DOTSI\upointop\ilimits@}
\makeatother

\newcommand{\qq}{\mathscr{Q}}
\newcommand{\RR}{\mathbb{R}}
\newcommand{\NN}{\mathbb{N}}

\newcommand{\HH}{\mathcal{H}}
\newcommand{\GG}{\mathcal{G}}

\newcommand{\BL}{\ensuremath{\EuScript B}\,}

\newcommand{\pinf}{{+}\infty}
\newcommand{\minf}{{-}\infty}
\newcommand{\zeroun}{\intv[o]{0}{1}}

\newcommand{\RXX}{\intv{\minf}{\pinf}}
\newcommand{\RX}{\intv[l]0{\minf}{\pinf}}

\newcommand{\RPP}{\intv[o]0{0}{\pinf}}

\newcommand{\infconv}{\mathbin{\mbox{\small$\square$}}}

\newcommand{\pushfwd}%
{\ensuremath{\mbox{\Large$\,\triangleright\,$}}}

\newcommand{\Id}{\mathrm{Id}}
\newcommand{\moyo}[2]{\leftindex[I]^{#2}{#1}}

\DeclareMathOperator{\dom}{dom}

\DeclareMathOperator{\rav}{{\mathsf{rav}}}

\DeclareFontFamily{U}{mathb}{}
\DeclareFontShape{U}{mathb}{m}{n}{<-5.5> mathb5 <5.5-6.5> mathb6 
<6.5-7.5> mathb7 <7.5-8.5> mathb8 <8.5-9.5> mathb9 <9.5-11> mathb10
<11-> mathb12}{}
\DeclareSymbolFont{mathb}{U}{mathb}{m}{n}
\DeclareFontSubstitution{U}{mathb}{m}{n}
\DeclareMathSymbol{\blackdiamond}{\mathbin}{mathb}{"0C}

\renewcommand{\leq}{\leqslant}
\renewcommand{\geq}{\geqslant}
\newcommand{\exi}{\exists\,}

\newcommand{\proxc}[1]{\mathbin%
{\ensuremath{\overset{#1}{\diamond}}}}
\newcommand{\proxcc}[1]{\mathbin%
{\ensuremath{\overset{#1}{\blackdiamond}}}}
\newcommand{\Rm}[1]{\overset{\mathord{\diamond}}{\mathsf{M}}_{#1}}
\newcommand{\Rcm}[1]{\overset{\mathord{\blackdiamond}}{\mathsf{M}}_{#1}}
\renewenvironment{abstract}{%
\vspace*{-0.50cm}
\small
\quotation%
\noindent%
{\normalfont\bfseries\sffamily
\nobreak\abstractname\ }%
}{%
\endquotation%
\medskip
}
\renewcommand{\abstractname}{Abstract.}

\newcommand\mscsname{MSC classification.}
\newenvironment{keywords}
{\renewcommand\abstractname{\keywordsname}\begin{abstract}}
{\end{abstract}}
\newenvironment{MSC}
{\renewcommand\abstractname{\mscsname}\begin{abstract}}
{\end{abstract}}
\usepackage[auth-sc]{authblk}
\newcommand{\email}[1]{\href{mailto:#1}{\nolinkurl{#1}}}
\renewcommand*\Affilfont{\normalfont\normalsize}
\newcommand\affilcr{\protect\\ \protect\Affilfont}
\makeatletter
\renewcommand\AB@affilsepx{\protect\\[0.5em]}
\makeatother

\author[]{Diego J. Cornejo}
\affil[]{North Carolina State University
\affilcr
Department of Mathematics
\affilcr
Raleigh, NC 27695, USA
\affilcr
\email{djcornej@ncsu.edu}
}

\begin{document}

\title{%
Resolvent Compositions\\for Positive Linear Operators
\thanks{%
This work was supported by the National
Science Foundation under grant CCF-2211123.
}}

\date{~}

\maketitle

\vspace{12mm}

\begin{abstract}
Resolvent compositions were recently introduced as
monotonicity-preserving operations that combine a set-valued
monotone operator and a bounded linear operator. They generalize in
particular the notion of a resolvent average. We analyze the
resolvent compositions when the monotone operator is a positive
linear operator. We establish several new properties, including
L\"{o}wner partial order relations, concavity, and asymptotic
behavior. In addition, we show that the resolvent composition
operations are nonexpansive with respect to the Thompson metric. We
also introduce a new form of geometric interpolation and explore
its connections to resolvent compositions. Finally, we study two
nonlinear equations based on resolvent compositions.
\end{abstract}


\begin{MSC}
47A63,
47A64,
47H05,
47H09
\end{MSC}

\newpage

\section{Introduction}
\label{sec:1}

Throughout, $\HH$ is a real Hilbert space with identity operator
$\Id_{\HH}$, scalar product $\scal{\cdot}{\cdot}_{\HH}$, and
associated norm $\|\cdot\|_{\HH}$. In addition, $\GG$ is a real
Hilbert space, the set of bounded linear operators from $\HH$ to
$\GG$ is denoted by $\BL(\HH,\GG)$, and $\BL(\HH)=\BL(\HH,\HH)$.
The adjoint of $L\in\BL(\HH,\GG)$ is denoted by $L^*$. The set
$\EuScript{P}(\HH)$ of positive operators on $\HH$ is the
collection of self-adjoint operators $A\in\BL(\HH)$ such that
$(\forall x\in\HH)$ $\scal{Ax}{x}\geq 0$. The L\"{o}wner partial
ordering between two self-adjoint operators $A$ and $B$ in
$\BL(\HH)$ is defined by 
$A\preccurlyeq B\Leftrightarrow B-A\in\EuScript{P}(\HH)$, and the
set of self-adjoint strongly monotone operators on $\HH$ is
\begin{equation}
\EuScript{S}(\HH)=\menge{A\in\EuScript{P}(\HH)}{
(\exi\alpha\in\RPP)\;\;\alpha\Id_{\HH}\preccurlyeq A}.
\end{equation}

A fundamental operator associated with a monotone operator
$B\colon\GG\to2^\GG$ is its resolvent
\begin{equation}
\label{e:res}
J_B=(\Id_{\GG}+B)^{-1},
\end{equation}
which plays a central role in monotone operator theory and convex
optimization, especially through its use in operator splitting
algorithms \cite{Livre1,Acnu24,Sipr21}. In many applications,
monotone operators arise in combination with linear operators,
which motivates the study of operations that combine the
monotone operator $B$ and a linear operator $L\in\BL(\HH,\GG)$
while preserving monotonicity.
Recently, \cite{Svva23} introduced two monotonicity-preserving
operations called the \emph{resolvent composition} and
the \emph{resolvent cocomposition} of $B$ and $L$, defined
respectively by
\begin{equation}
\label{e:rescomp}
L\proxc{\gamma}B
=L^*\pushfwd\brk1{B+\gamma^{-1}\Id_{\GG}}-\gamma^{-1}\Id_{\HH}
\end{equation}
and 
\begin{equation}
\label{e:rescoco}
L\proxcc{\gamma}B=\brk2{L\proxc{1/\gamma}B^{-1}}^{-1},
\end{equation}
where $L^*\pushfwd B=(L^*\circ B^{-1}\circ L)^{-1}$ is the
\emph{parallel composition} of $B$ by $L^*$ \cite{Livre1}, and 
$\gamma\in\RPP$. 
An attractive property of resolvent compositions is that their
resolvent operators can be computed explicitly 
\cite[Propositions~1.2 and 4.1(v)]{Svva23}, namely,
\begin{equation}
\label{e:resres}
J_{\gamma\brk1{L\proxc{\gamma}B}}
=L^*\circ J_{\gamma B}\circ L
\quad\text{and}\quad
J_{\gamma\brk1{L\proxcc{\gamma}B}}
=\Id_{\HH}-L^*\circ\brk1{\Id_{\GG}-J_{\gamma B}}\circ L.
\end{equation}
This feature, in turn, significantly facilitates the
design and implementation of algorithms for monotone inclusion and
convex optimization problems
\cite{Jota24,Svva23,Acnu24,Eusi24,Svva25}. Special cases can also
be implicitly found in concrete applications such as image recovery
\cite{Eusi24}, neural networks \cite{Hasa20}, inverse problems
\cite{Kami17}, and machine learning \cite{Shen17,Yuyl13}. For
further motivation, let us consider the following examples.

\begin{example}[resolvent mixtures]
\label{ex:comix}
Let $0\neq p\in\NN$ and let $\gamma\in\RPP$. For every
$k\in\{1,\ldots,p\}$, let $\GG_k$ be a real Hilbert space, let
$L_k\in\BL(\HH,\GG_k)$ be such that $0<\|L_k\|\leq 1$, let
$B_k\in\EuScript{S}(\GG_k)$, and let $\alpha_k\in\RPP$.
Suppose that $\sum_{k=1}^p\alpha_k=1$, let
$\GG=\bigoplus_{k=1}^p\GG_k$, and set
\begin{equation}
L\colon\HH\to\GG\colon x\mapsto
\brk1{\sqrt{\alpha_k}L_kx}_{1\leq k\leq p}\quad\text{and}\quad
B\colon\GG\to\GG\colon(y_k)_{1\leq k\leq p}\mapsto
(B_ky_k)_{1\leq k\leq p}.
\end{equation}
Then $L\proxc{\gamma}B=\Rm{\gamma}(L_k,B_k)_{1\leq k\leq p}$ and
$L\proxcc{\gamma}B=\Rcm{\gamma}(L_k,B_k)_{1\leq k\leq p}$ are
called \emph{resolvent mixture} and \emph{resolvent comixture},
respectively, introduced in \cite{Svva23} and subsequently studied
in \cite{Jota24,Svva25}. Further, as shown in
\cite[Proposition~5.13(i)]{Svva25}, 
$\Rcm{\gamma}(L_k,B_k)_{1\leq k\leq p}$ graph converges to 
$\sum_{k=1}^p\alpha_kL_k^*\circ B_k\circ L_k$ as $\gamma\to 0$.
\end{example}

\begin{example}[arithmetic, harmonic, and resolvent average]
\label{ex:rav}
In the context of Example~\ref{ex:comix}, suppose that, for every
$k\in\{1,\ldots,p\}$, $\GG_k=\HH$ and $L_k=\Id_{\HH}$. Then, the
\emph{arithmetic average} and the \emph{harmonic average} are
given, respectively, by
\begin{equation}
\label{e:1}
L^*\circ B\circ L=\sum_{k=1}^p\alpha_kB_k
\quad\text{and}\quad
L^*\pushfwd B=\brk3{\sum_{k=1}^p\alpha_kB_k^{-1}}^{-1}.
\end{equation}
An alternative averaging operation is the \emph{resolvent average},
introduced in \cite{Baus10} and further studied in
\cite{Baus16,Svva23,Wang11}, given by
\begin{equation}
\label{e:rav}
\rav_{\gamma}(B_k)_{1\leq k\leq p}=\brk3{\sum_{k=1}^p\alpha_k
\brk2{B_k+\gamma^{-1}\Id_{\HH}}^{-1}}^{-1}-\gamma^{-1}\Id_{\HH},
\quad\text{where}\quad\gamma\in\RPP.
\end{equation}
The resolvent average is as a special case of the resolvent
mixtures \cite[Example~1.3]{Svva23}, to wit,
\begin{equation}
\Rcm{\gamma}(\Id_k,B_k)_{1\leq k\leq p}
=\Rm{\gamma}(\Id_k,B_k)_{1\leq k\leq p}
=\rav_{\gamma}(B_k)_{1\leq k\leq p}.
\end{equation}
It was observed in \cite{Kum15,Laws14} that, when 
$\EuScript{S}(\HH)$ is endowed with the \emph{Thompson
metric} $d_{T}^{\HH}$ (see \eqref{e:thomp}), the averages
\eqref{e:1} and \eqref{e:rav} are nonexpansive. This property is
important as it ensures stability of the averaging processes and is
particularly useful in the study of nonlinear equations
\cite{Kum15}. Further, in the finite-dimensional setting,
\cite[Corollary~4.6]{Baus10} shows that the resolvent average is
concave, and \cite[Theorem~4.2]{Baus10} establishes, by means of a
pointwise convergence proof, that the resolvent average
interpolates between the arithmetic average ($0<\gamma\to0$) and
the harmonic average ($\gamma\to\pinf$).
\end{example}

\begin{example}[weighted $\mathscr{A}\#\mathscr{H}$--means]
\label{ex:ah}
In the finite-dimensional setting of Example~\ref{ex:rav}, a family
$\brk1{\mathcal{L}_{\gamma}(B_k)_{1\leq k\leq p}}_{\gamma\in\RR}$
of means interpolating between the arithmetic average and the
harmonic average was introduced in \cite{Kim11} (see also
\cite{Kim22}). These means, referred to as \emph{weighted
$\mathscr{A}\#\mathscr{H}$--means}, are closely related to
resolvent averages \cite[Proposition~3.5]{Kim11} through the
ordering
\begin{equation}
\rav_{\gamma}(B_k)_{1\leq k\leq p}\preccurlyeq
\mathcal{L}_{1/\gamma}(B_k)_{1\leq k\leq p}\preccurlyeq
\sum_{k=1}^p\alpha_kB_k,
\end{equation}
and themselves interpolate between the arithmetic
average ($\gamma\to\pinf$) and the harmonic average
($\gamma\to\minf$) \cite[Proposition~3.4]{Kim11}.
\end{example}

The aim of this paper is to investigate the operations
\eqref{e:rescomp} and \eqref{e:rescoco} when
$B\in\EuScript{S}(\GG)$. We establish several new properties,
including L\"{o}wner partial order relations, concavity,
nonexpansiveness, and asymptotic behavior. This specific setting
leads to new results that, in particular, generalize the
corresponding asymptotic properties in \cite{Svva25}, as well as
those of the proximal average established in
\cite{Baus10,Kim11,Kum15}.

The remainder of the paper is organized as follows.
In Section~\ref{sec:2}, we provide our notation and necessary
mathematical background. In Section~\ref{sec:3}, we present
several new properties of $(L\proxcc{\gamma}B)_{\gamma\in\RPP}$ and
$(L\proxc{\gamma}B)_{\gamma\in\RPP}$. In particular, these
operations are concave and
\begin{itemize}
\item
$L\proxcc{\gamma}B\preccurlyeq L^*\circ B\circ L$ \quad and \quad
$L\proxcc{\gamma}B\to L^*\circ B\circ L$\;\;as\;\;$0<\gamma\to 0$,
\item
$L^*\pushfwd B\preccurlyeq L\proxc{\gamma}B$ \quad and \quad
$L\proxc{\gamma}B\to L^*\pushfwd B$\;\;as\;\;$\gamma\to\pinf$.
\end{itemize}
In Section~\ref{sec:4}, we show that the resolvent compositions
are nonexpansive with respect to the Thompson metric, in the sense
that, for every $A\in\EuScript{S}(\GG)$ and
$B\in\EuScript{S}(\GG)$,
\begin{equation}
d_{T}^{\HH}\brk1{L\proxcc{\gamma}A,L\proxcc{\gamma}B}\leq
d_{T}^{\GG}(A,B)
\quad\text{and}\quad
d_{T}^{\HH}\brk1{L\proxc{\gamma}A,L\proxc{\gamma}B}\leq
d_{T}^{\GG}(A,B).
\end{equation}
Finally, in Section~\ref{sec:5}, we introduce the geometric
interpolation $\mathcal{L}_{\gamma}(L,B)$ (see \eqref{e:p110})
between $L^*\pushfwd B$ and $L^*\circ B\circ L$ when $L$ is an
isometry, which generalizes the weighted
$\mathscr{A}\#\mathscr{H}$--means. We establish the partial order
relations
\begin{equation}
L^*\pushfwd B\preccurlyeq\mathcal{L}_{-\gamma}(L,B)\preccurlyeq
L\proxcc{\gamma}B\preccurlyeq\mathcal{L}_{1/\gamma}(L,B)
\preccurlyeq L^*\circ B\circ L,
\end{equation}
and conclude by studying two nonlinear equations involving
resolvent compositions.

\section{Notation and background}
\label{sec:2}

The space $\BL(\HH,\GG)$ is endowed with the topology induced by 
the operator norm
\begin{equation}
\label{e:opnorm}
\brk1{\forall L\in\BL(\HH,\GG)}\quad
\|L\|=\sup_{\substack{x\in\HH\\ \|x\|_{\HH}\leq 1}}\|Lx\|_{\GG}.
\end{equation}
Let $L\in\BL(\HH,\GG)$. Then $L$ is an isometry if 
$L^*\circ L=\Id_{\HH}$. Further, $L$ is bounded below if there
exists $\alpha\in\RPP$ such that $(\forall x\in\HH)$
$\alpha\|x\|_{\HH}\leq\|Lx\|_{\GG}$. Equivalently, by
\cite[Fact~2.26]{Livre1}, $L$ is bounded below if and only if
$L$ is injective with closed range. In particular, when $\HH$ and
$\GG$ are finite-dimensional, $L$ is bounded below if and only if
$\ker L=\{0\}$.

The quadratic kernel of $A\in\EuScript{P}(\HH)$ is
$\qq_A\colon\HH\to\RR\colon x\mapsto(1/2)\scal{x}{Ax}_{\HH}$. 
The Legendre conjugate of $f\colon\HH\to\RXX$ is the function 
\begin{equation}
f^*\colon\HH\to\RXX\colon x^*\mapsto
\sup_{x\in\HH}\brk1{\scal{x}{x^*}_{\HH}-f(x)},
\end{equation}
and the Moreau envelope of $f\colon\HH\to\RXX$ of parameter
$\gamma\in\RPP$ is
\begin{equation}
\moyo{f}{\gamma}\colon\HH\to\RXX\colon x\mapsto
\inf_{z\in\HH}\brk2{f(z)+\dfrac{1}{2\gamma}\|x-z\|_{\HH}^2}.
\end{equation}
The set of proper lower semicontinuous convex functions from $\HH$
to $\RX$ is denoted by $\Gamma_0(\HH)$. Let $L\in\BL(\HH,\GG)$ and
$h\colon\GG\to\RXX$. The infimal postcomposition of $h$ by $L^*$ is
\begin{equation}
L^*\pushfwd h\colon\HH\to\RXX\colon x\mapsto
\inf_{\substack{y\in\GG\\L^*y=x}}h(y),
\end{equation} 
the proximal composition of $h$ and $L$ with parameter
$\gamma\in\RPP$ (see \cite{Svva23,Eect25}) is
\begin{equation}
L\proxc{\gamma}h
=\brk2{\moyo{\brk1{h^*}}{\frac{1}{\gamma}}\circ L}^*
-\frac{1}{2\gamma}\|\cdot\|_{\HH}^2,
\end{equation}
and the proximal cocomposition of $h$ and $L$ with parameter
$\gamma\in\RPP$ is 
\begin{equation}
\label{e:pcc}
L\proxcc{\gamma}h
=\brk1{L\proxc{1/\gamma}h^*}^*.
\end{equation}

The following facts will be used subsequently.

\begin{lemma}
\label{l:1}
The following properties are satisfied:
\begin{enumerate}
\item
\label{l:1i}
Let $A\in\EuScript{S}(\GG)$. Then
$\qq_A^*=\qq_{A^{-1}}$.
\item
\label{l:1ii}
Let $A\in\EuScript{S}(\GG)$ and
$B\in\EuScript{S}(\GG)$. Then
$A\preccurlyeq B\Leftrightarrow B^{-1}\preccurlyeq A^{-1}$.
\item
\label{l:1iii}
Let $L\in\BL(\HH,\GG)$, $A\in\EuScript{P}(\GG)$, and
$B\in\EuScript{P}(\GG)$. Then
$A\preccurlyeq B\Rightarrow L^*\circ A\circ L\preccurlyeq 
L^*\circ B\circ L$.
\item
\label{l:1iv}
Let $A\in\EuScript{P}(\GG)$ and 
$B\in\EuScript{P}(\GG)$. Then
$A\preccurlyeq B\Rightarrow\|A\|\leq\|B\|$.
\item
\label{l:1v}
Let $(A_n)_{n\in\NN}$, $(B_n)_{n\in\NN}$, $A$, and $B$ be
self-adjoint operators in $\BL(\GG)$ such that $A_n\to A$,
$B_n\to B$, and $(\forall n\in\NN)$ $A_n\preccurlyeq B_n$. Then
$A\preccurlyeq B$.
\end{enumerate}
\end{lemma}
\begin{proof}
\ref{l:1i}--\ref{l:1ii}: See the proof of
\cite[Example~13.18(i)]{Livre1}.

\ref{l:1iii}: Let $x\in\HH$. Since $A\preccurlyeq B$, 
$\scal{x}{L^*(A(Lx))}=\scal{Lx}{A(Lx)}\leq\scal{Lx}{B(Lx)}
=\scal{x}{L^*(B(Lx))}$.

\ref{l:1iv}: Since $A$ and $B$ are self-adjoint and 
$0\preccurlyeq A\preccurlyeq B$, we deduce from
\cite[Fact~2.25(iii)]{Livre1} that
\begin{equation}
\|A\|=\sup_{\substack{x\in\GG\\\|x\|_{\GG}\leq 1}}\lvert
\scal{Ax}{x}_{\GG}\rvert
=\sup_{\substack{x\in\GG\\\|x\|_{\GG}\leq 1}}\scal{Ax}{x}_{\GG}\leq
\sup_{\substack{x\in\GG\\\|x\|_{\GG}\leq 1}}\scal{Bx}{x}_{\GG}
=\sup_{\substack{x\in\GG\\\|x\|_{\GG}\leq 1}}
\lvert\scal{Bx}{x}_{\GG}\rvert=\|B\|.
\end{equation}

\ref{l:1v}: Since $A_n\to A$ and $B_n\to B$, convergence is in
particular pointwise. Thus, for every $x\in\HH$,
$0\leq\scal{x}{(B_n-A_n)x}\to\scal{x}{(B-A)x}$. Hence,
$0\preccurlyeq B-A$ or, equivalently, $A\preccurlyeq B$.
\end{proof}

\begin{lemma}
\label{l:2}
Suppose that $L\in\BL(\HH,\GG)$ satisfies $0<\|L\|\leq 1$, let
$g\in\Gamma_0(\GG)$, and let $\gamma\in\RPP$. Then the following
hold:
\begin{enumerate}
\item
\label{l:2i}
$L\proxc{\gamma}g=(L\proxcc{1/\gamma}g^*)^*$.
\item
\label{l:2ii}
$L\proxcc{\gamma}g\leq\min\{\,L\proxc{\gamma}g\,,\,g\circ L\,\}$.
\item
\label{l:2iii}
Set $\Phi=(1/2)\|\cdot\|_{\GG}^2-(1/2)\|\cdot\|_{\HH}^2\circ L^*$.
Then $L\proxcc{\gamma}g=(g^*+\gamma\Phi)^*\circ L$.
\item
\label{l:2iv}
Set $\Phi=(1/2)\|\cdot\|_{\GG}^2-(1/2)\|\cdot\|_{\HH}^2\circ L^*$.
Then $L\proxc{\gamma}g=L^*\pushfwd(g+\Phi/\gamma)$.
\end{enumerate}
\end{lemma}
\begin{proof} 
Recall that $g=g^{**}$ \cite[Corollary~13.38]{Livre1}.

\ref{l:2i}: \cite[Proposition~3.7(iii)]{Eect25}.

\ref{l:2ii}: \cite[Proposition~3.20(ii)--(iii)]{Eect25}. 

\ref{l:2iii}--\ref{l:2iv}: \cite[Proposition~3.2(i)--(ii)]{Eect25}.
\end{proof}

\begin{lemma}
\label{l:3}
Let $L\in\BL(\HH,\GG)$ and $B\in\EuScript{P}(\GG)$.
Then the following hold:
\begin{enumerate}
\item
\label{l:3i}
$L^*\circ B\circ L\in\EuScript{P}(\HH)$.
\item
\label{l:3ii}
$\qq_B\circ L=\qq_{L^*\circ B\circ L}$.
\item
\label{l:3iii}
Suppose that $B\in\EuScript{S}(\GG)$ and that $L$ is bounded below.
Then $L^*\circ B\circ L\in\EuScript{S}(\HH)$ and
$L^*\pushfwd B\in\EuScript{S}(\HH)$.
\end{enumerate}
\end{lemma}
\begin{proof}
\ref{l:3i}: Take $A=0$ in Lemma~\ref{l:1}\ref{l:1iii}.

\ref{l:3ii}: For every $x\in\HH$, $\qq_B(Lx)=(1/2)\scal{Lx}{B(Lx)}
=(1/2)\scal{x}{L^*(B(Lx))}=\qq_{L^*\circ B\circ L}(x)$. 

\ref{l:3iii}: Since $B\in\EuScript{S}(\GG)$, there exists
$\alpha\in\RPP$ such that $\alpha\Id_{\GG}\preccurlyeq B$. On the
other hand, since $L$ is bounded below, there exists $\beta\in\RPP$
such that $\beta^2\Id_{\HH}\preccurlyeq L^*\circ L$. Therefore, 
Lemma~\ref{l:1}\ref{l:1iii} yields
\begin{equation}
(\alpha\beta^2)\Id_{\HH}\preccurlyeq\alpha(L^*\circ L)
=L^*\circ(\alpha\Id_{\GG})\circ L\preccurlyeq L^*\circ B\circ L,
\end{equation}
i.e., $L^*\circ B\circ L\in\EuScript{S}(\HH)$. Similarly, 
$L^*\circ B^{-1}\circ L\in\EuScript{S}(\HH)$, which implies that
$L^*\pushfwd B=(L^*\circ B^{-1}\circ L)^{-1}\in\EuScript{S}(\HH)$.
\end{proof}

\begin{lemma}[\protect{\cite[Proposition~3.3(ii)]{Svva25}}]
\label{l:4}
Let $L\in\BL(\HH,\GG)$, let $B\colon\GG\to2^\GG$, let 
$\gamma\in\RPP$, and set $\Psi=\Id_{\GG}-L\circ L^*$. Then
$L\proxcc{\gamma}B=L^*\circ(B^{-1}+\gamma\Psi)^{-1}\circ L$.
\end{lemma}

\begin{lemma}[\protect{\cite[Proposition~3.4(i)]{Svva25}}]
\label{l:5}
Suppose that $L\in\BL(\HH,\GG)$ is an isometry, let
$B\colon\GG\to2^\GG$, and let $\gamma\in\RPP$. Then 
$L\proxc{\gamma}B=L\proxcc{\gamma}B$.
\end{lemma}

\section{Resolvent compositions}
\label{sec:3}

In this section, we study the resolvent cocomposition operators
when $B\in\EuScript{S}(\GG)$. We strengthen several results
obtained in \cite{Svva25}, as well as those established
specifically for the resolvent average in \cite{Baus10}.
The results obtained include
comparisons among the different composite operations, as well as an
analysis of the asymptotic behavior of
$(L\proxcc{\gamma}B)_{\gamma\in\RPP}$ and
$(L\proxc{\gamma}B)_{\gamma\in\RPP}$, as the parameter $\gamma$
varies. 

\begin{proposition}
\label{p:29}
Suppose that $L\in\BL(\HH,\GG)$ satisfies $0<\|L\|\leq 1$, let
$B\in\EuScript{S}(\GG)$, and let $\gamma\in\RPP$. Then the
following hold:
\begin{enumerate}
\item
\label{p:29i}
$L\proxcc{\gamma}B\in\EuScript{P}(\HH)$.
\item
\label{p:29ii}
$L\proxcc{\gamma}\qq_B=\qq_{L\proxcc{\gamma}B}$.
\item
\label{p:29iii}
Let $\lambda\in\zeroun$. Then
$T_{\gamma}\colon\EuScript{S}(\GG)\to\EuScript{P}(\HH)\colon
A\mapsto L\proxcc{\gamma}A$ is concave in the sense that
\begin{equation}
\label{e:p29iii}
\brk1{\forall A\in\EuScript{S}(\GG)}\quad
\lambda\brk1{L\proxcc{\gamma}A}
+(1-\lambda)\brk1{L\proxcc{\gamma}B}\preccurlyeq
L\proxcc{\gamma}\brk1{\lambda A+(1-\lambda)B}.
\end{equation}
\item
\label{p:29iv}
Suppose that $L$ is bounded below. Then the following are
satisfied:
\begin{enumerate}
\item
\label{p:29iva}
$L\proxcc{\gamma}B\in\EuScript{S}(\HH)$ and
$L\proxc{\gamma}B\in\EuScript{S}(\HH)$.
\item
\label{p:29ivb}
$L\proxc{\gamma}\qq_B=\qq_{L\proxc{\gamma}B}$.
\item
\label{p:29ivc}
Let $\lambda\in\zeroun$. Then $R_{\gamma}\colon\EuScript{S}(\GG)\to
\EuScript{S}(\HH)\colon A\mapsto L\proxc{\gamma}A$ is concave in
the sense that
\begin{equation}
\label{e:p29ivc}
\brk1{\forall A\in\EuScript{S}(\GG)}\quad
\lambda\brk1{L\proxc{\gamma}A}
+(1-\lambda)\brk1{L\proxc{\gamma}B}\preccurlyeq
L\proxc{\gamma}\brk1{\lambda A+(1-\lambda)B}.
\end{equation}
\end{enumerate}
\end{enumerate}
\end{proposition}
\begin{proof}
Set $\Psi=\Id_{\GG}-L\circ L^*$. Since $\|L\|\leq 1$,
$\Psi\in\EuScript{P}(\GG)$, which yields
$B^{-1}+\gamma\Psi\in\EuScript{S}(\GG)$. On the other hand,
recall from Lemma~\ref{l:4} that
\begin{equation}
\label{e:p29}
L\proxcc{\gamma}B=L^*\circ\brk1{B^{-1}+\gamma\Psi}^{-1}\circ L.
\end{equation}

\ref{p:29i}: This follows from \eqref{e:p29} and 
Lemma~\ref{l:3}\ref{l:3i}.

\ref{p:29ii}: Set
$\Phi=(1/2)\|\cdot\|_{\GG}^2-(1/2)\|\cdot\|_{\HH}^2\circ L^*$ and
note that $\Phi=\qq_{\Psi}$. It follows from
Lemma~\ref{l:2}\ref{l:2iii}, Lemma~\ref{l:1}\ref{l:1i},
Lemma~\ref{l:3}\ref{l:3ii}, and \eqref{e:p29} that
\begin{align}
\label{e:200}
L\proxcc{\gamma}\qq_B&=\brk1{\qq_B^*+\gamma\Phi}^*\circ L
\nonumber\\
&=\brk1{\qq_{B^{-1}}+\gamma\qq_{\Psi}}^*\circ L\nonumber\\
&=\qq_{B^{-1}+\gamma\Psi}^*\circ L\nonumber\\
&=\qq_{L^*\circ\brk1{B^{-1}+\gamma\Psi}^{-1}\circ L}\nonumber\\
&=\qq_{L\proxcc{\gamma}B}.
\end{align}

\ref{p:29iii}:  
By \ref{p:29i}, $T_{\gamma}$ is well defined. Further, for every
$A\in\EuScript{S}(\GG)$, 
\begin{align}
\label{e:p29iiib}
&\lambda\brk1{L\proxcc{\gamma}A}
+(1-\lambda)\brk1{L\proxcc{\gamma}B}\preccurlyeq
L\proxcc{\gamma}\brk1{\lambda A+(1-\lambda)B}\nonumber\\
&\Leftrightarrow\brk1{\forall x\in\HH}\,\,
\lambda\Scal{\brk1{L\proxcc{\gamma}A}x}{x}_{\HH}
+(1-\lambda)\Scal{\brk1{L\proxcc{\gamma}B}x}{x}_{\HH}
\leq\Scal{
\brk2{L\proxcc{\gamma}\brk1{\lambda A+(1-\lambda)B}}x}{x}_{\HH}
\nonumber\\
&\Leftrightarrow\brk1{\forall x\in\HH}\,\,
\lambda\qq_{L\proxcc{\gamma}A}(x)
+(1-\lambda)\qq_{L\proxcc{\gamma}B}(x)
\leq\qq_{\lambda\brk1{L\proxcc{\gamma}A}
+(1-\lambda)\brk1{L\proxcc{\gamma}B}}(x).
\end{align}
Therefore, it is enough to prove that, for every
$x\in\HH$, the function $\EuScript{S}(\GG)\to\RR\colon A\mapsto
\qq_{L\proxcc{\gamma}A}(x)$ is concave. Set
$\Phi=(1/2)\|\cdot\|_{\GG}^2-(1/2)\|\cdot\|_{\HH}^2\circ L^*$.
Because $\dom\Phi=\GG$, the identity $(\gamma\Phi)^*=\Phi^*/\gamma$
and \cite[Proposition~15.2]{Livre1} imply that 
\begin{equation}
\label{e:conc1}
\brk1{\forall A\in\EuScript{S}(\GG)}\quad
\brk1{\qq_A^*+\gamma\Phi}^*=\qq_A\infconv\brk1{\Phi^*/\gamma}\colon
\GG\to\RX\colon z\mapsto\inf_{y\in\GG}\brk2{\qq_A(y)
+\dfrac{1}{\gamma}\Phi^*(z-y)}.
\end{equation}
Thus, by virtue of \ref{p:29ii}, Lemma~\ref{l:2}\ref{l:2iii}, and
\eqref{e:conc1},
\begin{align}
\brk1{\forall A\in\EuScript{S}(\GG)}
\brk1{\forall x\in\HH}\quad
\qq_{L\proxcc{\gamma}A}(x)&=\brk1{L\proxcc{\gamma}\qq_A}(x)
\nonumber\\
&=\brk1{\qq_A^*+\gamma\Phi}^*(Lx)\nonumber\\
&=\inf_{y\in\GG}\brk2{\underbrace{\qq_A(y)
+\dfrac{1}{\gamma}\Phi^*(Lx-y)}_{\text{affine in}\;\;A}}.
\end{align}
Hence, for every $x\in\HH$, the function 
$\EuScript{S}(\GG)\to\RR\colon
A\mapsto\qq_{L\proxcc{\gamma}A}(x)$ is concave, as it can be
expressed as the infimum of affine functions.

\ref{p:29iva}: It follows from \eqref{e:p29} and
Lemma~\ref{l:3}\ref{l:3iii} that
$L\proxcc{\gamma}B\in\EuScript{S}(\HH)$. On the other hand,
by \eqref{e:rescoco} and applying the previous reasoning to
$B^{-1}$, we obtain 
$L\proxc{\gamma}B=(L\proxcc{1/\gamma}B^{-1})^{-1}\in
\EuScript{S}(\HH)$.

\ref{p:29ivb}: By Lemma~\ref{l:2}\ref{l:2i},
Lemma~\ref{l:1}\ref{l:1i}, \ref{p:29ii}, and \eqref{e:rescoco},
\begin{equation}
L\proxc{\gamma}\qq_B=\brk1{L\proxcc{1/\gamma}\qq_B^*}^*
=\brk1{L\proxcc{1/\gamma}\qq_{B^{-1}}}^*
=\qq_{L\proxcc{1/\gamma}B^{-1}}^*
=\qq_{\brk1{L\proxcc{1/\gamma}B^{-1}}^{-1}}
=\qq_{L\proxc{\gamma}B}.
\end{equation}

\ref{p:29ivc}: It follows from Lemma~\ref{l:2}\ref{l:2iv} and
\ref{p:29ivb} that
\begin{equation}
\brk1{\forall A\in\EuScript{S}(\GG)}\brk1{\forall x\in\GG}
\quad \qq_{L\proxc{\gamma}A}(x)
=\brk1{L\proxc{\gamma}\qq_A}(x)
=\inf_{\substack{y\in\GG\\L^*y=x}}\brk2{
\underbrace{\qq_A(y)+\dfrac{1}{\gamma}\Phi(y)}_{
\text{affine in}\;\;A}}
\end{equation}
Thus, for every $x\in\HH$, the function
$\EuScript{S}(\GG)\to\RR\colon
A\mapsto\qq_{L\proxc{\gamma}A}(x)$ is concave. As a consequence, as
in the proof of \ref{p:29iii}, $R_{\gamma}$ is concave.
\end{proof}

The following example shows that, in the finite-dimensional
setting, the resolvent composition admits a variational
characterization. In particular, this holds for the resolvent
average, as established in \cite[Proposition~2.8]{Baus10}.

\begin{example}[variational characterization]
\label{ex:varc}
Suppose that $\HH$ and $\GG$ are finite-dimensional and that
$L\in\BL(\HH,\GG)$ satisfies $\|L\|\leq 1$ and $\ker L=\{0\}$,
let $B\in\EuScript{S}(\GG)$, and let $\gamma\in\RPP$. Define
\begin{equation}
\label{e:varf}
f\colon\EuScript{S}(\HH)\to\RR\colon X\mapsto
-\ln\det\brk1{X+\gamma^{-1}\Id_{\HH}}
\end{equation}
and 
\begin{equation}
\label{e:varF}
F\colon\EuScript{S}(\HH)\to\RR\colon X\mapsto 
f(X)+\Scal{L^*\circ(B+\gamma^{-1}\Id_{\GG})^{-1}\circ L}{X},
\end{equation}
where $\det(X)$ denotes the determinant of $X$ and 
$\scal{X}{B}$ denotes the trace of $X\circ B$.
Then $L\proxc{\gamma}B$ is the unique minimizer of $F$.
\begin{proof}
Since $\EuScript{S}(\HH)\to\EuScript{S}(\HH)\colon X\mapsto
X+\gamma^{-1}\Id_{\HH}$ is affine, 
\cite[Example~24.66 and Proposition~8.20]{Livre1} show that
$f$ is convex, differentiable, and that
$(\forall X\in\EuScript{S}(\HH))$ 
$\nabla f(X)=-(X+\gamma^{-1}\Id_{\HH})^{-1}$. Thus,
$F$ is also convex and differentiable, being the sum of $f$ and an
affine function. Therefore, by virtue of \cite[Theorem~16.3 and
Proposition~17.31(i)]{Livre1}, it suffices to find the critical
points of $F$, that is, to solve $\nabla F(X)=0$. Altogether,
Proposition~\ref{p:29}\ref{p:29iva} ensures that
$L\proxc{\gamma}B\in\EuScript{S}(\HH)$, and
\begin{align}
\nabla F(X)=0&\Leftrightarrow-\brk1{X+\gamma^{-1}\Id_{\HH}}^{-1}
+L^*\circ\brk1{B+\gamma^{-1}\Id_{\GG}}^{-1}\circ L=0\nonumber\\
&\Leftrightarrow X+\gamma^{-1}\Id_{\HH}
=L^*\pushfwd(B+\gamma^{-1}\Id_{\GG})\nonumber\\
&\Leftrightarrow X=L\proxc{\gamma}B,
\end{align}
which completes the proof.
\end{proof}
\end{example}

We now focus on L\"{o}wner partial ordering relations for resolvent
compositions. These ordering relations will assist us in studying
the convergence properties of resolvent compositions
$L\proxc{\gamma}B$ and $L\proxcc{\gamma}B$, as well as of the new
interpolation $\mathcal{L}_{\gamma}(L,B)$ introduced in
Section~\ref{sec:5}, as $\gamma$ varies. 

\begin{proposition}
\label{p:30}
Suppose that $L\in\BL(\HH,\GG)$ satisfies $0<\|L\|\leq 1$, let
$B\in\EuScript{S}(\GG)$, and let $\gamma\in\RPP$. Then the
following hold:
\begin{enumerate}
\item
\label{p:30i}
Set $\theta=1/(1+\gamma\|B\|)$. Then
$\theta(L^*\circ B\circ L)\preccurlyeq L\proxcc{\gamma}B
\preccurlyeq L^*\circ B\circ L$.
\item
\label{p:30ii}
Suppose that $A\in\EuScript{S}(\GG)$ satisfies
$A\preccurlyeq B$. Then
$L\proxcc{\gamma}A\preccurlyeq L\proxcc{\gamma}B$.
\item
\label{p:30iii}
Let $\rho\in\RPP$ be such that $\rho\leq\gamma$. Then
$L\proxcc{\gamma}B\preccurlyeq L\proxcc{\rho}B$.
\end{enumerate}
\end{proposition}
\begin{proof}
Set $\Psi=\Id_{\GG}-L\circ L^*$ and recall that
$L\proxcc{\gamma}B=L^*\circ(B^{-1}+\gamma\Psi)^{-1}\circ L$ by
Lemma~\ref{l:4}.

\ref{p:30i}: Note that $B\preccurlyeq\|B\|\;\Id_{\GG}$ and that
Lemma~\ref{l:1}\ref{l:1ii} implies that
$\Id_{\GG}\preccurlyeq\|B\|\;B^{-1}$. Since
$0\preccurlyeq\Psi\preccurlyeq\Id_{\GG}$,
\begin{equation}
B^{-1}\preccurlyeq B^{-1}+\gamma\Psi\preccurlyeq
B^{-1}+\gamma\Id_{\GG}\preccurlyeq(1+\gamma\|B\|)\;B^{-1},
\end{equation}
and, by virtue of Lemma~\ref{l:1}\ref{l:1ii},
\begin{equation}
\label{e:p30}
\theta B\preccurlyeq\brk1{B^{-1}+\gamma\Psi}^{-1}\preccurlyeq B.
\end{equation}
Hence, we deduce from \eqref{e:p30} and Lemma~\ref{l:1}\ref{l:1iii}
that
\begin{align}
\theta(L^*\circ B\circ L)\preccurlyeq L\proxcc{\gamma}B\preccurlyeq
L^*\circ B\circ L.
\end{align}

\ref{p:30ii}: Since $\Psi\in\EuScript{P}(\GG)$,
$A^{-1}+\gamma\Psi$ and $B^{-1}+\gamma\Psi$ are in
$\EuScript{S}(\GG)$. Further, by Lemma~\ref{l:1}\ref{l:1ii} and the
fact that $A\preccurlyeq B$,
$B^{-1}+\gamma\Psi\preccurlyeq A^{-1}+\gamma\Psi$. Thus,
$(A^{-1}+\gamma\Psi)^{-1}\preccurlyeq(B^{-1}+\gamma\Psi)^{-1}$.
Altogether, we deduce from Lemma~\ref{l:1}\ref{l:1iii} that
\begin{equation}
L\proxcc{\gamma}A=L^*\circ\brk1{A^{-1}+\gamma\Psi}^{-1}\circ L
\preccurlyeq L^*\circ\brk1{B^{-1}+\gamma\Psi}^{-1}\circ L
=L\proxcc{\gamma}B.
\end{equation}

\ref{p:30iii}: Note that $B^{-1}+\gamma\Psi$ and $B^{-1}+\rho\Psi$
are in $\EuScript{S}(\GG)$ and that
$B^{-1}+\rho\Psi\preccurlyeq B^{-1}+\gamma\Psi$. Therefore,
Lemma~\ref{l:1}\ref{l:1ii}-\ref{l:1iii} yields
\begin{align}
L\proxcc{\gamma}B=L^*\circ\brk1{B^{-1}+\gamma\Psi}^{-1}\circ L
\preccurlyeq L^*\circ\brk1{B^{-1}+\rho\Psi}^{-1}\circ L
=L\proxcc{\rho}B,
\end{align}
as claimed.
\end{proof}

\begin{corollary}
\label{c:31}
Suppose that $L\in\BL(\HH,\GG)$ is bounded below and satisfies
$\|L\|\leq 1$, let $B\in\EuScript{S}(\GG)$, and let
$\gamma\in\RPP$. Then the following hold:
\begin{enumerate}
\item
\label{c:31i}
Set $\omega=1+\|B^{-1}\|/\gamma$. Then
$L^*\pushfwd B\preccurlyeq L\proxc{\gamma}B\preccurlyeq
\omega(L^*\pushfwd B)$.
\item
\label{c:31ii}
$L\proxcc{\gamma}B\preccurlyeq L\proxc{\gamma}B$.
\item
\label{c:31iii}
Suppose that $A\in\EuScript{S}(\GG)$ satisfies
$A\preccurlyeq B$. Then
$L\proxc{\gamma}A\preccurlyeq L\proxc{\gamma}B$.
\item
\label{c:31iv}
Let $\rho\in\RPP$ be such that $\rho\leq\gamma$. Then
$L\proxc{\gamma}B\preccurlyeq L\proxc{\rho}B$.
\end{enumerate}
\end{corollary}
\begin{proof}
By Proposition~\ref{p:29}\ref{p:29iva},
$L\proxc{\gamma}B\in\EuScript{S}(\HH)$. Further, recall that
\eqref{e:rescoco} yields
$L\proxc{\gamma}B=(L\proxcc{1/\gamma}B^{-1})^{-1}$.

\ref{c:31i}: This follows from Lemma~\ref{l:1}\ref{l:1ii} and
Proposition~\ref{p:30}\ref{p:30i} applied to $B^{-1}$ and
$1/\gamma$.

\ref{c:31ii}: By Proposition~\ref{p:29}\ref{p:29ii},
Lemma~\ref{l:2}\ref{l:2ii}, and
Proposition~\ref{p:29}\ref{p:29ivb},
\begin{equation}
\qq_{L\proxcc{\gamma}B}=L\proxcc{\gamma}\qq_B\leq 
L\proxc{\gamma}\qq_B=\qq_{L\proxc{\gamma}B}.
\end{equation}
Therefore, $L\proxcc{\gamma}B\preccurlyeq
L\proxc{\gamma}B$.

\ref{c:31iii}: This follows from Lemma~\ref{l:1}\ref{l:1ii} and
Proposition~\ref{p:30}\ref{p:30ii} applied to $B^{-1}$ and
$1/\gamma$.

\ref{c:31iv}: This follows from Lemma~\ref{l:1}\ref{l:1ii} and
Proposition~\ref{p:30}\ref{p:30iii} applied to $B^{-1}$ and
$1/\gamma$.
\end{proof}

\begin{corollary}
Suppose that $L\in\BL(\HH,\GG)$ is bounded below and satisfies
$\|L\|\leq 1$, let $B\in\EuScript{S}(\GG)$, and set
$\kappa=\|B\|\;\|B^{-1}\|$ and
$\rho=(1+\sqrt{\kappa})^2$. Then
$L^*\circ B\circ L\preccurlyeq\rho(L^*\pushfwd B)$.
\end{corollary}
\begin{proof}
Set $f\colon\RPP\to\RPP\colon\gamma\to
(1+\gamma\|B\|)(1+\|B^{-1}\|/\gamma)$. By 
Proposition~\ref{p:30}\ref{p:30i},
Corollary~\ref{c:31}\ref{c:31ii}, and
Corollary~\ref{c:31}\ref{c:31i},
\begin{equation}
\label{e:c32}
(\forall\gamma\in\RPP)\quad
L^*\circ B\circ L\preccurlyeq f(\gamma)\brk1{L^*\pushfwd B}.
\end{equation}
Since $\rho=\min_{\gamma\in\RPP}f(\gamma)$, the assertion follows
from \eqref{e:c32}. 
\end{proof}

We now present the main result of this section. In contrast to
\cite[Propositions~5.8 and 5.12(i)]{Svva25}, which establish
graph convergence of resolvent compositions, the following theorem
provides asymptotic behavior of resolvent compositions in operator
norm, which is stronger than graph convergence and therefore
offers additional stability properties.

\begin{theorem}
\label{t:53}
Suppose that $L\in\BL(\HH,\GG)$ satisfies $0<\|L\|\leq 1$, and let
$B\in\EuScript{S}(\GG)$. Then the following hold:
\begin{enumerate}
\item
\label{t:53i}
$L\proxcc{\gamma}B\to L^*\circ B\circ L$ as $0<\gamma\to 0$.
\item
\label{t:53ii}
Suppose that $L$ is bounded below. Then
$L\proxc{\gamma}B\to L^*\pushfwd B$ as $\gamma\to\pinf$.
\end{enumerate}
\end{theorem}
\begin{proof}
\ref{t:53i}: Set $(\forall\gamma\in\RPP)$ 
$\theta_{\gamma}=1/(1+\gamma\|B\|)$ and 
$D_{\gamma}=(L^*\circ B\circ L)-(L\proxcc{\gamma}B)$. By
Proposition~\ref{p:30}\ref{p:30i},
\begin{equation}
\label{e:t53a}
0\preccurlyeq D_{\gamma}\preccurlyeq
\brk2{\dfrac{1-\theta_{\gamma}}{\theta_{\gamma}}}\;
(L^*\circ B\circ L).
\end{equation}
In addition, note that $\theta_{\gamma}\to 1$ as $0<\gamma\to 0$.
Therefore, it follows from \eqref{e:t53a} and
Lemma~\ref{l:1}\ref{l:1iv} that 
\begin{equation}
\|D_{\gamma}\|\leq
\brk2{\dfrac{1-\theta_{\gamma}}{\theta_{\gamma}}}
\;\|L^*\circ B\circ L\|\to 0\;\;\text{as}\;\;
0<\gamma\to 0.
\end{equation}

\ref{t:53ii}: Set $(\forall\gamma\in\RPP)$
$\omega_{\gamma}=1+\|B^{-1}\|/\gamma$ and
$D_{\gamma}=(L\proxc{\gamma}B)-(L^*\pushfwd B)$. By
Corollary~\ref{c:31}\ref{c:31i},
\begin{equation}
\label{e:t53b}
0\preccurlyeq D_{\gamma}\preccurlyeq
(\omega_{\gamma}-1)\;(L^*\pushfwd B).
\end{equation}
Also, note that $\omega_{\gamma}\to 1$ as $\gamma\to\pinf$.
Therefore, we combine \eqref{e:t53b} and Lemma~\ref{l:1}\ref{l:1iv}
to obtain
\begin{equation}
\|D_{\gamma}\|\leq
(\omega_{\gamma}-1)\;\|L^*\pushfwd B\|\to 0\;\;\text{as}\;\;
0<\gamma\to\pinf,
\end{equation}
which completes the proof.
\end{proof}

\begin{corollary}
\label{c:100}
Suppose that $L\in\BL(\HH,\GG)$ is bounded below and satisfies 
$\|L\|\leq 1$. Then the operator
$R\colon\EuScript{S}(\GG)\to\EuScript{S}(\HH)\colon A\mapsto 
L^*\pushfwd A$ is concave in the sense that
\begin{equation}
\label{e:c100a}
\brk1{\forall\lambda\in\zeroun}
\brk1{\forall A\in\EuScript{S}(\GG)}
\brk1{\forall B\in\EuScript{S}(\GG)}\quad
\lambda(L^*\pushfwd A)+(1-\lambda)(L^*\pushfwd B)\preccurlyeq
L^*\pushfwd\brk1{\lambda A+(1-\lambda)B}.
\end{equation} 
\end{corollary}
\begin{proof}
By Proposition~\ref{p:29}\ref{p:29ivc},
$R_{\gamma}\colon\EuScript{S}(\GG)\to\EuScript{S}(\HH)\colon
A\mapsto L\proxc{\gamma}A$ is concave, i.e.,
\begin{equation}
\label{e:c100b}
\brk1{\forall\lambda\in\zeroun}
\brk1{\forall A\in\EuScript{S}(\GG)}
\brk1{\forall B\in\EuScript{S}(\GG)}\quad
\lambda(L\proxc{\gamma}A)+(1-\lambda)(L\proxc{\gamma}B)\preccurlyeq
L^*\proxc{\gamma}\brk1{\lambda A+(1-\lambda)B}.
\end{equation} 
Hence, letting $\gamma\to\pinf$ in \eqref{e:c100b} and
invoking Theorem~\ref{t:53}\ref{t:53ii} together with
Lemma~\ref{l:1}\ref{l:1v}, we obtain \eqref{e:c100a}.
\end{proof}

\begin{remark}
\label{r:100}
In the context of the resolvent averages of Example~\ref{ex:rav},
Theorem~\ref{t:53} and Corollary~\ref{c:100} generalize
\cite[Theorem~4.2 and Corollary~4.6]{Baus10}, which were
established in the finite-dimensional context using different
techniques.
\end{remark}

\begin{corollary}
\label{c:54}
Suppose that $L\in\BL(\HH,\GG)$ is an isometry, and let
$B\in\EuScript{S}(\GG)$. Then the
following hold:
\begin{enumerate}
\item
\label{c:54i}
$(\forall\gamma\in\RPP)$ $L^*\pushfwd B\preccurlyeq
L\proxcc{\gamma}B\preccurlyeq L^*\circ B\circ L$.
\item
\label{c:54ii}
$L\proxcc{\gamma}B\to L^*\circ B\circ L$ as $0<\gamma\to 0$.
\item
\label{c:54iii}
$L\proxcc{\gamma}B\to L^*\pushfwd B$ as $\gamma\to\pinf$.
\end{enumerate}
\end{corollary}
\begin{proof}
Since $L$ is an isometry, Lemma~\ref{l:5} yields
$L\proxc{\gamma}B=L\proxcc{\gamma}B$.

\ref{c:54i}: This follows from Proposition~\ref{p:30}\ref{p:30i}
and Corollary~\ref{c:31}\ref{c:31i}.

\ref{c:54ii}: This follows from Theorem~\ref{t:53}\ref{t:53i}.

\ref{c:54iii}: This follows from Theorem~\ref{t:53}\ref{t:53ii}.
\end{proof}

\begin{corollary}[resolvent mixtures]
\label{c:35}
Consider the setting of Example~\ref{ex:comix}. Then the following
hold:
\begin{enumerate}
\item
\label{c:35i}
$\Rcm{\gamma}(L_k,B_k)_{1\leq k\leq p}\preccurlyeq
\sum_{k=1}^p\alpha_kL_k^*\circ B_k\circ L_k$.
\item
\label{c:35ii}
$\Rcm{\gamma}(L_k,B_k)_{1\leq k\leq p}\to
\sum_{k=1}^p\alpha_kL_k^*\circ B_k\circ L_k$ as
$0<\gamma\to 0$. 
\item 
\label{c:35iii}
Suppose that $L_j$ is bounded below for some $j\in\{1,\ldots,p\}$.
Then the following are satisfied:
\begin{enumerate}
\item
\label{c:35iiia}
$\Rm{\gamma}(L_k,B_k)_{1\leq k\leq p}\in\EuScript{S}(\HH)$
and
$\Rcm{\gamma}(L_k,B_k)_{1\leq k\leq p}\in\EuScript{S}(\HH)$.
\item
\label{c:35iiib}
$\brk1{\sum_{k=1}^p\alpha_kL_k^*\circ B_k^{-1}\circ L_k}^{-1}
\preccurlyeq\Rm{\gamma}(L_k,B_k)_{1\leq k\leq p}$. 
\item
\label{c:35iiic}
$\Rm{\gamma}(L_k,B_k)_{1\leq k\leq p}\to
\brk1{\sum_{k=1}^p\alpha_kL_k^*\circ B_k^{-1}\circ L_k}^{-1}$
as $\gamma\to\pinf$. 
\end{enumerate}
\end{enumerate}
\end{corollary}
\begin{proof}
Note that $L^*\circ B\circ L=\sum_{k=1}^p\alpha_lL_k^*\circ B_k
\circ L_k$ and $L^*\pushfwd B=\brk1{\sum_{k=1}^p\alpha_kL_k^*\circ
B_k^{-1}\circ L_k}^{-1}$. Further, if $L_j$ is bounded below for
some $j\in\{1,\ldots,p\}$, then $L$ is also bounded below. Indeed,
there exists $\alpha\in\RPP$ such that $(\forall x\in\HH)$ 
$\alpha\|x\|_{\HH}\leq\|L_jx\|_{\GG_j}$. Thus, $L$ is bounded below
since
\begin{equation}
(\forall x\in\HH)\quad\|Lx\|_{\GG}
=\brk3{\sum_{k=1}^p\alpha_k\|L_kx\|_{\GG_k}^2}^{1/2}
\geq\brk2{\alpha_j\alpha^2\|x\|_{\HH}^2}^{1/2}
=\brk1{\alpha_j^{1/2}\alpha}\|x\|_{\HH}.
\end{equation}

\ref{c:35i}: This follows from Proposition~\ref{p:30}\ref{p:30i}.

\ref{c:35ii}: This follows from Theorem~\ref{t:53}\ref{t:53i}.

\ref{c:35iiia}: This follows from 
Proposition~\ref{p:29}\ref{p:29iva}.

\ref{c:35iiib}: This follows from Corollary~\ref{c:31}\ref{c:31i}.

\ref{c:35iiic}: This follows from Theorem~\ref{t:53}\ref{t:53ii}.
\end{proof}

\begin{corollary}
\label{c:37}
Consider the setting of Example~\ref{ex:rav}. Then the following
hold:
\begin{enumerate}
\item
\label{c:37i}
$\brk1{\sum_{k=1}^p\alpha_kB_k^{-1}}^{-1}\preccurlyeq
\rav_{\gamma}(B_k)_{1\leq k\leq p}\preccurlyeq
\sum_{k=1}^p\alpha_kB_k$.
\item
\label{c:37ii}
$\rav_{\gamma}(B_k)_{1\leq k\leq p}\to
\sum_{k=1}^p\alpha_kB_k$ as $0<\gamma\to 0$.
\item
\label{c:37iii}
$\rav_{\gamma}(B_k)_{1\leq k\leq p}\to
\brk1{\sum_{k=1}^p\alpha_kB_k^{-1}}^{-1}$ as $\gamma\to\pinf$.
\end{enumerate}
\end{corollary}
\begin{proof}
Recall that $\rav_{\gamma}(B_k)_{1\leq k\leq p}
=\Rm{\gamma}(L_k,B_k)_{1\leq k\leq p}
=\Rcm{\gamma}(L_k,B_k)_{1\leq k\leq p}$.

\ref{c:37i}: This follows from items \ref{c:35i} and \ref{c:35iiib}
in Corollary~\ref{c:35}.

\ref{c:37ii}: This follows from Corollary~\ref{c:35}\ref{c:35ii}.

\ref{c:37iii}: This follows from
Corollary~\ref{c:35}\ref{c:35iiic}.
\end{proof}

\section{Nonexpansiveness of resolvent compositions}
\label{sec:4}

In this section, we build on the results of Section~\ref{sec:3} to 
prove that the resolvent composition operations are nonexpansive
with respect to the Thompson metric \cite{Thomp63} on 
$\EuScript{S}(\HH)$, defined by
\begin{equation}
\label{e:thomp}
\brk1{\forall A\in\EuScript{S}(\HH)}
\brk1{\forall B\in\EuScript{S}(\HH)}\quad
d_{T}^{\HH}(A,B)=\ln\brk1{\max\{g(A,B),g(B,A)\}},
\end{equation}
where $g(A,B)
=\inf\menge{\lambda\in\RPP}{A\preccurlyeq\lambda B}$.\\
The Thompson metric was originally defined on cones in Banach
spaces \cite{Thomp63}.
Since $\EuScript{S}(\HH)$ is contained in the cone of
monotone self-adjoint operators on $\BL(\HH)$, which
is closed and hence complete in the operator norm topology, and
since every $A\in\EuScript{S}(\HH)$
satisfies $\alpha\Id_{\HH}\preccurlyeq
A\preccurlyeq\|A\|\Id_{\HH}$ for some $\alpha\in\RPP$, it follows
from \cite[Lemma~3]{Thomp63} that
$(\EuScript{S}(\HH),d_{T}^{\HH})$ is a complete metric space.
The metric $d_{T}^{\HH}$ provides a geometric structure on
$\EuScript{S}(\HH)$ that plays a central role in the study of
nonlinear matrix equations, especially for establishing existence
and uniqueness results via Banach contraction mappings
\cite{Lee08,Lemm15,Lim09,Lim12}, and in various applications
to nonlinear optimization \cite{Gaub14,Laws07,Mont98}. 
In this context, the nonexpansiveness of resolvent compositions is
crucial, as it ensures that the resulting operations preserve both
the metric structure and the stability necessary for analysis. For
instance, in Section~\ref{sec:5}, we present two
nonlinear equations based on resolvent compositions that admit
unique solutions.

\begin{theorem}
\label{t:thomp}
Suppose that $L\in\BL(\HH,\GG)$ is bounded below and satisfies
$\|L\|\leq 1$, and let $\gamma\in\RPP$. Then the following hold:
\begin{enumerate}
\item
\label{t:thompi}
$T_{\gamma}\colon\brk1{\EuScript{S}(\GG),d_{T}^{\GG}}
\to\brk1{\EuScript{S}(\HH),d_{T}^{\HH}}\colon B\mapsto 
L\proxcc{\gamma}B$ is nonexpansive, i.e.,
\begin{equation}
\brk1{\forall A\in\EuScript{S}(\GG)}
\brk1{\forall B\in\EuScript{S}(\GG)}\quad
d_{T}^{\HH}\brk1{L\proxcc{\gamma}A,L\proxcc{\gamma}B}\leq
d_{T}^{\GG}(A,B).
\end{equation}
\item
\label{t:thompii}
$R_{\gamma}\colon\brk1{\EuScript{S}(\GG),d_{T}^{\GG}}
\to\brk1{\EuScript{S}(\HH),d_{T}^{\HH}}\colon B\mapsto 
L\proxc{\gamma}B$ is nonexpansive, i.e.,
\begin{equation}
\brk1{\forall A\in\EuScript{S}(\GG)}
\brk1{\forall B\in\EuScript{S}(\GG)}\quad
d_{T}^{\HH}\brk1{L\proxc{\gamma}A,L\proxc{\gamma}B}\leq
d_{T}^{\GG}(A,B).
\end{equation}
\end{enumerate}
\end{theorem}
\begin{proof}
Let $A$ and $B$ be in $\EuScript{S}(\GG)$, and set
$g(A,B)=\inf\menge{\lambda\in\RPP}{A\preccurlyeq\lambda B}$.

\ref{t:thompi}:
Note that the operator $T_{\gamma}$ is well defined by
Proposition~\ref{p:29}\ref{p:29iva}. By virtue of \eqref{e:thomp},
\begin{equation}
\label{e:thomp2}
A\preccurlyeq e^{d_{T}^{\GG}(A,B)}B.
\end{equation}
On the other hand, it follows from 
\cite[Proposition~3.1(vi)]{Svva25} and
Proposition~\ref{p:30}\ref{p:30iii} that
\begin{equation}
\label{e:thomp3}
(\forall\rho\in\intv[r]{1}{\pinf})\quad
L\proxcc{\gamma}(\rho B)=\rho\brk1{L\proxcc{\gamma\rho}B}
\preccurlyeq\rho\brk1{L\proxcc{\gamma}B}.
\end{equation}
Since $e^{d_{T}^{\GG}(A,B)}\geq 1$, we combine
Proposition~\ref{p:30}\ref{p:30ii}, \eqref{e:thomp2}, and
\eqref{e:thomp3} to obtain
\begin{equation}
\label{e:thomp4}
L\proxcc{\gamma}A\preccurlyeq 
L\proxcc{\gamma}\brk1{e^{d_{T}^{\GG}(A,B)}B}
\preccurlyeq e^{d_{T}^{\GG}(A,B)}\brk1{L\proxcc{\gamma}B}.
\end{equation}
In turn, 
\begin{equation}
\label{e:thomp5}
g\brk1{L\proxcc{\gamma}A,L\proxcc{\gamma}B}
=\inf\menge{\lambda\in\RPP}{L\proxcc{\gamma}A\preccurlyeq
\lambda(L\proxcc{\gamma}B)}\leq e^{d_{T}^{\GG}(A,B)}.
\end{equation}
By the same argument,
\begin{equation}
\label{e:thomp6}
g\brk1{L\proxcc{\gamma}B,L\proxcc{\gamma}A}\leq
e^{d_{T}^{\GG}(A,B)}.
\end{equation}
Altogether, it follows from \eqref{e:thomp}, \eqref{e:thomp5}, and
\eqref{e:thomp6} that
\begin{equation}
d_{T}^{\HH}\brk1{L\proxcc{\gamma}A,L\proxcc{\gamma}B}
=\max\bigl\{\,\ln g\brk1{L\proxcc{\gamma}A,L\proxcc{\gamma}B}\,,\,
\ln g\brk1{L\proxcc{\gamma}B,L\proxcc{\gamma}A}\,\bigr\}
\leq d_{T}^{\GG}(A,B).
\end{equation}

\ref{t:thompii}: 
Note that $R_{\gamma}$ is well defined by
Proposition~\ref{p:29}\ref{p:29iva}. Since
$d_{T}^{\GG}(A,B)=d_{T}^{\GG}(A^{-1},B^{-1})$, we deduce from
\ref{t:thompi} and \eqref{e:rescoco} that
\begin{equation}
d_{T}^{\HH}\brk1{L\proxc{\gamma}A,L\proxc{\gamma}B}=
d_{T}^{\HH}\brk1{L\proxcc{1/\gamma}A^{-1},L\proxcc{1/\gamma}B^{-1}}
\leq d_{T}^{\GG}(A^{-1},B^{-1})=d_{T}^{\GG}(A,B),
\end{equation}
as announced.
\end{proof}

\begin{corollary}
\label{c:thomp}
Consider the setting of Example~\ref{ex:comix}. Suppose that
$L_j$ is bounded below for some $j\in\{1,\ldots,p\}$ and that, for
every $k\in\{1,\ldots,p\}$, $A_k\in\EuScript{S}(\GG_k)$, and set
$A\colon\GG\to\GG\colon(y_k)_{1\leq k\leq p}\mapsto
(A_ky_k)_{1\leq k\leq p}$. Then
\begin{equation}
d_{T}^{\HH}\brk2{\Rm{\gamma}(L_k,A_k)_{1\leq k\leq p},
\Rm{\gamma}(L_k,B_k)_{1\leq k\leq p}}\leq d_{T}^{\GG}(A,B)
=\max_{1\leq k\leq p}d_{\GG_k}(A_k,B_k)
\end{equation}
and
\begin{equation}
d_{T}^{\HH}\brk2{\Rcm{\gamma}(L_k,A_k)_{1\leq k\leq p},
\Rcm{\gamma}(L_k,B_k)_{1\leq k\leq p}}\leq d_{T}^{\GG}(A,B)
=\max_{1\leq k\leq p}d_{\GG_k}(A_k,B_k).
\end{equation}
In other words, the resolvent mixtures are nonexpansive for the
Thompson metric.
\end{corollary}
\begin{proof}
It is straightforward to verify that 
$d_{T}^{\GG}(A,B)=\displaystyle\max_{1\leq k\leq p}
d_{\GG_k}(A_k,B_k)$. On the other hand, 
$L\proxc{\gamma}A=\Rm{\gamma}(L_k,A_k)_{1\leq k\leq p}$ and
$L\proxcc{\gamma}A=\Rcm{\gamma}(L_k,A_k)_{1\leq k\leq p}$.
Hence, the assertion follows from Theorem~\ref{t:thomp}.
\end{proof}

\begin{corollary}[\protect{\cite[Theorem~3.5]{Kum15}}]
Consider the setting of Example~\ref{ex:rav}.
Suppose that, for every $k\in\{1,\ldots,p\}$,
$A_k\in\EuScript{S}(\HH)$, and set $A\colon\GG\to\GG\colon
(y_k)_{1\leq k\leq p}\mapsto(A_ky_k)_{1\leq k\leq p}$. Then
\begin{equation}
d_{T}^{\HH}\brk1{\rav_{\gamma}(A_k)_{1\leq k\leq p},
\rav_{\gamma}(B_k)_{1\leq k\leq p}}\leq d_{T}^{\GG}(A,B).
\end{equation}
In other words, the resolvent average is nonexpansive for the
Thompson metric.
\end{corollary}
\begin{proof}
Since $\rav_{\gamma}(A_k)_{1\leq k\leq p}
=\Rcm{\gamma}(\Id_{\HH},A_k)_{1\leq k\leq p}$,
the conclusion follows from Corollary~\ref{c:thomp}.
\end{proof}

\section{Geometric means and nonlinear equations}
\label{sec:5}

Let $A\in\EuScript{S}(\GG)$. Since $A$ is strongly monotone, there
exists $\alpha\in\RPP$ such that $\alpha\Id_{\GG}\preccurlyeq A$.
Consequently, the spectrum $\sigma(A)$ is contained in the compact
interval $[\alpha,\|A\|]$
(see \cite[Theorem~VI.6 and Problem~VII.12]{RedSim}). 
Hence, for every $t\in\RR$, the function
$f_t\colon\sigma(A)\to\RR\colon\lambda\to\lambda^t$ is well defined
and continuous, and the operator $f_t(A)=A^t$ is therefore defined
according to the continuous functional calculus 
(see \cite[Section~VII]{RedSim}).

Given $A\in\EuScript{S}(\GG)$ and $B\in\EuScript{S}(\GG)$, an
important instance of Kubo-Ando's operator means \cite{Kubo80} is
the $t$-\emph{weighted geometric mean}
\cite{Ando87,Kum15,Laws14,Lim12} of $A$ and $B$, defined by 
\begin{equation}
\label{e:geomean}
(\forall t\in[0,1])\quad
A\#_tB
=A^{1/2}\circ\brk2{A^{-1/2}\circ B\circ A^{-1/2}}^t\circ A^{1/2}.
\end{equation}
From a geometric viewpoint, the curve $t\mapsto A\#_tB$ describes
a minimal geodesic between $A$ and $B$ with respect to the 
Thompson metric (see, e.g., \cite[Lemma~2.2(iv)]{Laws14}), in the
sense that
\begin{equation}
\label{e:52}
(\forall t\in[0,1])(\forall s\in[0,1])\quad
d_{T}^{\GG}(A\#_tB,A\#_sB)=\lvert t-s\rvert\,d_{T}^{\GG}(A,B).
\end{equation}
In particular, the geometric mean
$A\#B=A\#_{1/2}B$ is the metric midpoint of the arithmetic mean
\mbox{$(A+B)/2$} and the harmonic mean $2(A^{-1}+B^{-1})^{-1}$ for
the Thompson metric (see \cite{Cora94,Laws01}). 

The following result introduces a new interpolation between
$L^*\pushfwd B$ and $L^*\circ B\circ L$, which generalizes the
weighted $\mathscr{A}\#\mathscr{H}$--mean discussed in
Example~\ref{ex:ah}.

\begin{proposition}
\label{p:110}
Suppose that $L\in\BL(\HH,\GG)$ is an isometry, let
$B\in\EuScript{S}(\GG)$, and let $\gamma\in\RPP$. Define
\begin{equation}
\label{e:p110}
\mathcal{L}_{\gamma}(L,B)
=\brk1{L^*\circ(B+\gamma\Id_{\GG})\circ L}\#\brk1{L^*\pushfwd(B
+\gamma\Id_{\GG})}-\gamma\Id_{\HH}
\end{equation}
and 
\begin{equation}
\mathcal{L}_{-\gamma}(L,B)
=\brk2{\mathcal{L}_{\gamma}\brk1{L,B^{-1}}}^{-1}.
\end{equation}
Then the following hold:
\begin{enumerate}
\item
\label{p:110i}
$L^*\pushfwd B\preccurlyeq\mathcal{L}_{-\gamma}(L,B)\preccurlyeq
L\proxcc{\gamma}B\preccurlyeq\mathcal{L}_{1/\gamma}(L,B)
\preccurlyeq L^*\circ B\circ L$.
\item
\label{p:110ii}
$\mathcal{L}_{\gamma}(L,B)\to L^*\circ B\circ L$ as
$\gamma\to\pinf$.
\item
\label{p:110iii}
$\mathcal{L}_{\gamma}(L,B)\to L^*\pushfwd B$ as
$\gamma\to\minf$.
\end{enumerate}
\end{proposition}
\begin{proof}
\ref{p:110i}:
Since $L$ is an isometry, $L^*\circ L=\Id_{\HH}$ and
Lemma~\ref{l:5} yields $L\proxc{\gamma}B=L\proxcc{\gamma}B$.
By Corollary~\ref{c:54}\ref{c:54i}, \eqref{e:p110}, and the fact
that $B\#B=B$, 
\begin{align}
\label{e:p110b}
\mathcal{L}_{1/\gamma}(L,B)
&\preccurlyeq\brk1{L^*\circ(B+\gamma^{-1}\Id_{\GG})\circ L}\#
\brk1{L^*\circ(B+\gamma^{-1}\Id_{\GG})\circ L}-\gamma^{-1}\Id_{\HH}
\nonumber\\&=\brk1{L^*\circ(B+\gamma^{-1}\Id_{\GG})\circ L}
-\gamma^{-1}\Id_{\HH}\nonumber\\
&=L^*\circ B\circ L+\gamma^{-1}(L^*\circ L-\Id_{\HH})\nonumber\\
&=L^*\circ B\circ L.
\end{align}
Similarly, \eqref{e:rescomp}, Corollary~\ref{c:54}\ref{c:54i}, and
\eqref{e:p110},
imply that
\begin{align}
\label{e:p110c}
L\proxc{\gamma}B&=L^*\pushfwd(B+\gamma^{-1}\Id_{\GG})
-\gamma^{-1}\Id_{\HH}\nonumber\\
&=\brk1{L^*\pushfwd(B+\gamma^{-1}\Id_{\GG})}\#
\brk1{L^*\pushfwd(B+\gamma^{-1}\Id_{\GG})}-\gamma^{-1}\Id_{\HH}
\nonumber\\
&\preccurlyeq\brk1{L^*\circ(B+\gamma^{-1}\Id_{\GG})\circ L}\#
\brk1{L^*\pushfwd(B+\gamma^{-1}\Id_{\GG})}-\gamma^{-1}\Id_{\HH}
\nonumber\\
&=\mathcal{L}_{1/\gamma}(L,B).
\end{align}
Thus, \eqref{e:p110b} and \eqref{e:p110c} yield 
\begin{equation}
\label{e:p110d}
L\proxcc{\gamma}B\preccurlyeq\mathcal{L}_{1/\gamma}(L,B)
\preccurlyeq L^*\circ B\circ L.
\end{equation} 
On the other hand, by virtue of Lemma~\ref{l:1}\ref{l:1ii},
\eqref{e:p110d} applied to $B^{-1}$ and $1/\gamma$, \eqref{e:p110},
and \eqref{e:rescoco},
\begin{align}
\label{e:p110e}
L^*\pushfwd B=(L^*\circ B^{-1}\circ L)^{-1}
\preccurlyeq\mathcal{L}_{\gamma}(L,B^{-1})^{-1}
=\mathcal{L}_{-\gamma}(L,B)
\preccurlyeq\brk1{L\proxcc{1/\gamma}B^{-1}}^{-1}
=L\proxc{\gamma}B=L\proxcc{\gamma}B.
\end{align}
Hence, the result follows from \eqref{e:p110d} and
\eqref{e:p110e}.

\ref{p:110ii}: This follows from \ref{p:110i} and
Corollary~\ref{c:54}\ref{c:54ii}.

\ref{p:110iii}: This follows from \ref{p:110i} and
Corollary~\ref{c:54}\ref{c:54iii}.
\end{proof}

\begin{remark}
\label{r:ah}
Note that the operator $\mathcal{L}_{\gamma}(L,B)$ is a type
of weighted geometric mean that interpolates between the parallel
composition $L^*\pushfwd B$ ($\gamma\to\minf$) and  
$L^*\circ B\circ L$ ($\gamma\to\pinf$). In the particular case
where $L$ and $B$ are defined as in Example~\ref{ex:rav}, 
$L^*\circ B\circ L=\sum_{k=1}^p\alpha_kB_k$ is the arithmetic
average, $L^*\pushfwd B =\brk1{\sum_{k=1}^p\alpha_kB_k^{-1}}^{-1}$
is the harmonic average, and $\mathcal{L}_{\gamma}(L,B)$ reduces
to the \emph{weighted $\mathscr{A}\#\mathscr{H}$-mean}
with parameter $\gamma$ of Example~\ref{ex:ah}, with
Proposition~\ref{p:110}\ref{p:110ii}--\ref{p:110iii} recovering
\cite[Proposition~3.4]{Kim11}.
\end{remark}

We now focus on nonlinear equations that are based on
resolvent compositions. The nonexpansive nature of these
operations, as shown in Section~\ref{sec:4}, will play a key role
in our subsequent analysis.

\begin{proposition}
\label{p:56}
Suppose that $L\in\BL(\HH,\GG)$ is bounded below and satisfies
$\|L\|\leq 1$, let $B\in\EuScript{S}(\GG)$, let $\gamma\in\RPP$,
and let $t\in\zeroun$. Set
\begin{equation}
\label{e:p56}
\varphi\colon\brk1{\EuScript{S}(\GG),d_{T}^{\GG}}\to
\brk1{\EuScript{S}(\HH),d_{T}^{\HH}}
\colon X\mapsto L\proxcc{\gamma}(X\#_tB).
\end{equation}
Then the following hold:
\begin{enumerate}
\item
\label{p:56i}
$\varphi$ is $(1-t)$-Lipschitzian. 
\item
\label{p:56ii}
Suppose that $\HH=\GG$. Then the problem
\begin{equation}
\label{e:p56b}
\text{find}\quad X\in\EuScript{S}(\HH)\quad\text{such that}
\quad X=L\proxcc{\gamma}(X\#_tB)
\end{equation}
admits a unique solution.
\end{enumerate}
\end{proposition}
\begin{proof}
\ref{p:56i}: It follows from Theorem~\ref{t:thomp}\ref{t:thompi}
and \cite[Lemma~2.2(iii)]{Laws14} that
\begin{align}
\brk1{\forall X\in\EuScript{S}(\GG)}
\brk1{\forall Y\in\EuScript{S}(\GG)}\quad
d_{T}^{\HH}\brk1{\varphi(X),\varphi(Y)}&=
d_{T}^{\HH}\brk1{L\proxcc{\gamma}(X\#_tB),L\proxcc{\gamma}(Y\#_tB)}
\nonumber\\
&\leq d_{T}^{\GG}(X\#_tB,Y\#_tB)\nonumber\\
&\leq(1-t)d_{T}^{\GG}(X,Y)+td_{T}^{\GG}(B,B)\nonumber\\
&=(1-t)d_{T}^{\GG}(X,Y).
\end{align}

\ref{p:56ii}: Since $d_{T}^{\HH}$ is a complete metric on
$\EuScript{S}(\HH)$, \ref{p:56i} and the Banach--Picard theorem
\cite[Theorem~1.50]{Livre1} ensure that $\varphi$ admits a
unique fixed point, i.e., \eqref{e:p56b} admits a unique solution.
\end{proof}

\begin{remark}
\label{r:56}
Let $X\in\EuScript{S}(\HH)$ be the unique solution
to \eqref{e:p56b}. Since $(X\#_tB)^{-1}=X^{-1}\#_tB^{-1}$ and
$L\proxcc{\gamma}B=(L\proxc{1/\gamma}B^{-1})^{-1}$, we note
that $X^{-1}$ is the unique solution to the
problem
\begin{equation}
\text{find}\quad Y\in\EuScript{S}(\HH)\quad\text{such that}
\quad Y=L\proxc{1/\gamma}(Y\#_tB^{-1}).
\end{equation}
\end{remark}

\begin{proposition}
\label{p:57}
Suppose that $L\in\BL(\HH,\GG)$ is bounded below and satisfies
$\|L\|\leq 1$, let $B\in\BL(\GG)$, let $\gamma\in\RPP$, and
let $t\in\intv[o]{-1}{1}$. Suppose that there exists a sequence
$(B_n)_{n\in\NN}$ of invertible operators in $\BL(\GG)$ such that
$B_n\to B$, and set 
\begin{equation}
\label{e:p57}
\varphi\colon\brk1{\EuScript{S}(\GG),d_{T}^{\GG}}\to
\brk1{\EuScript{S}(\HH),d_{T}^{\HH}}\colon X\mapsto 
L\proxcc{\gamma}(B^*\circ X^t\circ B).
\end{equation}
Then the following hold:
\begin{enumerate}
\item
\label{p:57i}
$\varphi$ is $|t|$-Lipschitzian.
\item
\label{p:57ii}
Suppose that $\HH=\GG$. Then the problem
\begin{equation}
\text{find}\quad X\in\EuScript{S}(\HH)\quad\text{such that}
\quad X=L\proxcc{\gamma}(B^*\circ X^t\circ B)
\end{equation}
admits a unique solution.
\end{enumerate}
\end{proposition}
\begin{proof}
\ref{p:57i}: 
Let $X\in\EuScript{S}(\GG)$ and
$Y\in\EuScript{S}(\GG)$. By \cite[Lemma~2.2(i)]{Laws14},
\begin{equation}
\label{e:p57a}
(\forall n\in\NN)\quad
d_{T}^{\GG}(B_n^*\circ X^t\circ B_n,B_n^*\circ Y^t\circ B_n)
=d_{T}^{\GG}(X^t,Y^t)
=d_{T}^{\GG}(X^{\lvert t\rvert},Y^{\lvert t\rvert})
\end{equation}
Thus, combining
Theorem~\ref{t:thomp}\ref{t:thompi}, \eqref{e:p57a}, and 
\cite[Lemma~2.2(iii)]{Laws14},
\begin{align}
\label{e:p57b}
\brk1{\forall n\in\NN}\quad
d_{T}^{\HH}\brk1{L\proxcc{\gamma}(B^*_n\circ X^t\circ B_n),
L\proxcc{\gamma}(B^*_n\circ Y^t\circ B_n)}
&\leq d_{T}^{\GG}(B^*_n\circ X^t\circ B_n,B^*_n\circ Y^t\circ B_n)
\nonumber\\
&=d_{T}^{\GG}(X^{\lvert t\rvert},Y^{\lvert t\rvert})\nonumber\\
&=d_{T}^{\GG}(\Id_{\GG}\#_{\lvert t\rvert}X,
\Id_{\GG}\#_{\lvert t\rvert}Y)\nonumber\\
&\leq\lvert t\rvert d_{T}^{\GG}(X,Y).
\end{align}
Altogether, by Theorem~\ref{t:thomp}\ref{t:thompi} and
\eqref{e:p57b}, we deduce that $(\forall
X\in\EuScript{S}(\GG))$
$(\forall Y\in\EuScript{S}(\GG))$,
\begin{align}
d_{T}^{\HH}\brk1{\varphi(X),\varphi(Y)}
&\leq d_{T}^{\HH}\brk1{\varphi(X),
L\proxcc{\gamma}(B^*_n\circ X^t\circ B_n)}
+d_{T}^{\HH}\brk1{L\proxcc{\gamma}(B^*_n\circ X^t\circ B_n),
L\proxcc{\gamma}(B^*_n\circ Y^t\circ B_n)}
\nonumber\\
&\quad+d_{T}^{\HH}\brk1{L\proxcc{\gamma}(B^*_n\circ Y^t\circ B_n),
\varphi(Y)}\nonumber\\
&\leq d_{T}^{\HH}\brk1{\varphi(X),L\proxcc{\gamma}(
B^*_n\circ X^t\circ B_n)}+\lvert t\rvert d_{T}^{\GG}(X,Y)
+d_{T}^{\HH}\brk1{L\proxcc{\gamma}(B^*_n\circ Y^t\circ B_n),
\varphi(Y)}\nonumber\\
&\leq d_{T}^{\GG}\brk1{B^*\circ X^t\circ B,B_n^*\circ X^t\circ B_n}
+\lvert t\rvert d_{T}^{\GG}(X,Y)
+d_{T}^{\GG}\brk1{B_n^*\circ Y^t\circ B_n,B^*\circ Y^t\circ B}
\nonumber\\
&\to\lvert t\rvert d_{T}^{\GG}(X,Y).
\end{align}

\ref{p:57ii}: This follows from \ref{p:57i} and the 
Banach--Picard theorem. 
\end{proof}

\begin{corollary}
\label{c:60}
Consider the setting of Example~\ref{ex:comix}. Suppose that
$L_j$ is bounded below for some $j\in\{1,\ldots,p\}$ and that, for
every $k\in\{1,\ldots,p\}$, $\GG_k=\HH$, and let $s\in\zeroun$ and
$t\in\intv[o]{-1}{1}$. Then the problems
\begin{equation}
\label{e:comix1}
\text{find}\quad X\in\EuScript{S}(\HH)\quad\text{such that}
\quad X=\Rcm{\gamma}\brk1{L_k,X\#_sB_k}_{1\leq k\leq p}
\end{equation}
and
\begin{equation}
\label{e:comix2}
\text{find}\quad X\in\EuScript{S}(\HH)\quad
\text{such that}\quad 
X=\Rcm{\gamma}\brk1{L_k,B_k^*\circ X^t\circ B_k}_{1\leq k\leq p}
\end{equation}
admit unique solutions. 
\end{corollary}
\begin{proof}
Set $R\colon\EuScript{S}(\HH)\to\EuScript{S}(\GG)
\colon X\mapsto \boldsymbol{\mathcal{X}}$, where
$\boldsymbol{\mathcal{X}}\colon\GG\to\GG\colon(y_k)\mapsto
(Xy_k)_{1\leq k\leq p}$, and set
\begin{equation}
\varphi_1\colon\brk1{\EuScript{S}(\GG),d_{T}^{\GG}}\to
\brk1{\EuScript{S}(\HH),d_{T}^{\HH}}
\colon X\mapsto L\proxcc{\gamma}(X\#_sB)
\end{equation}
and
\begin{equation}
\varphi_2\colon\brk1{\EuScript{S}(\GG),d_{T}^{\GG}}\to
\brk1{\EuScript{S}(\HH),d_{T}^{\HH}}\colon X\mapsto 
L\proxcc{\gamma}(B^*\circ X^t\circ B).
\end{equation}
Note that $(\forall\lambda\in\RPP)$
$X\preccurlyeq\lambda Y\Rightarrow
\boldsymbol{\mathcal{X}}\preccurlyeq
\lambda\boldsymbol{\mathcal{Y}}$. Thus, 
\begin{equation}
\label{e:c60a}
g(\boldsymbol{\mathcal{X}},\boldsymbol{\mathcal{Y}})
=\inf\menge{\lambda\in\RPP}{\boldsymbol{\mathcal{X}}\preccurlyeq
\lambda\boldsymbol{\mathcal{Y}}}
\leq\inf\menge{\lambda\in\RPP}{X\preccurlyeq\lambda Y}=g(X,Y),
\end{equation}
and it follows from \eqref{e:thomp} that
\begin{equation}
\label{e:c60b}
d_{T}^{\GG}\brk1{R(X),R(Y)}
=d_{T}^{\GG}(\boldsymbol{\mathcal{X}},\boldsymbol{\mathcal{Y}})
\leq d_{T}^{\HH}(X,Y).
\end{equation}
Now, given that $R$ is nonexpansive,
Propositions~\ref{p:56}\ref{p:56i} implies that $\varphi_1\circ R$
is $(1-s)$-Lipschitzian, whereas Proposition~\ref{p:57}\ref{p:57i}
implies that $\varphi_2\circ R$ is $\lvert t\rvert$-Lipschitzian.
Further, since $\boldsymbol{\mathcal{X}}\#_sB\colon\GG\to\GG\colon
(y_k)_{1\leq k\leq p}\mapsto\brk1{(X\#_sB_k)y_k}_{1\leq k\leq p}$ 
and $B^*\circ X^t\circ B\colon\GG\to\GG\colon(y_k)_{1\leq k\leq p}
\mapsto\brk1{(B_k^*\circ X^t\circ B_k)y_k}_{1\leq k\leq p}$, we
deduce that
\begin{equation}
\varphi_1\circ T\colon\EuScript{S}(\HH)\to
\EuScript{S}(\HH)\colon X\mapsto
L\proxcc{\gamma}(\boldsymbol{\mathcal{X}}\#_sB)
=\Rcm{\gamma}(L_k,X\#_sB_k)_{1\leq k\leq p}
\end{equation}
and
\begin{equation}
\varphi_2\circ T\colon\EuScript{S}(\HH)\to
\EuScript{S}(\HH)\colon X\mapsto
L\proxcc{\gamma}(B^*\circ X^t\circ B)
=\Rcm{\gamma}(L_k,B_k^*\circ X^t\circ B_k)_{1\leq k\leq p}.
\end{equation}
Altogether, it follows from the Banach--Picard theorem that
$\varphi_1\circ T$ and $\varphi_2\circ T$ admit unique fixed
points, i.e., the problems \eqref{e:comix1} and \eqref{e:comix2}
admit unique solutions.
\end{proof}

\begin{corollary}[\protect{\cite[Theorem~4.2]{Kum15}}]
Consider the setting of Example~\ref{ex:rav}, and let $s\in\zeroun$
and $t\in\intv[o]{-1}{1}$. Then the problems
\begin{equation}
\label{e:rav1}
\text{find}\quad X\in\EuScript{S}(\HH)\quad\text{such that}
\quad X=\rav_\gamma(X\#_sB_k)_{1\leq k\leq p}
\end{equation}
and 
\begin{equation}
\label{e:rav2}
\text{find}\quad X\in\EuScript{S}(\HH)\quad\text{such that}
\quad X=\rav_\gamma(B_k^*\circ X^t\circ B_k)_{1\leq k\leq p}
\end{equation}
admit unique solutions.
\end{corollary}
\begin{proof}
A direct consequence of Corollary~\ref{c:60}.
\end{proof}

\begin{remark}
According to Corollary~\ref{c:35}\ref{c:35ii}, the limit problems
of \eqref{e:comix1} and \eqref{e:comix2} as $0<\gamma\to 0$ are
\begin{equation}
\label{e:ave1}
\text{find}\quad X\in\EuScript{S}(\HH)\quad\text{such that}
\quad X=\sum_{k=1}^p\alpha_kL_k^*\circ(X\#_sB_k)\circ L_k
\end{equation}
and
\begin{equation}
\label{e:ave2}
\text{find}\quad X\in\EuScript{S}(\HH)\quad\text{such that}
\quad X=\sum_{k=1}^p\alpha_kL_k^*\circ(B_k^*\circ X^t\circ B_k)
\circ L_k.
\end{equation}
These problems and the uniqueness of their solutions were studied
in \cite{Laws14,Lim09,Lim12} when, for every $k\in\{1,\ldots,p\}$,
$\GG_k=\HH$ and $L_k=\Id_{\HH}$.
\end{remark}

\section*{Acknowledgement}

This work is a part of the author's Ph.D. dissertation. The author
acknowledges the guidance of his Ph.D. advisor P. L. Combettes,
throughout this work. 


\begin{thebibliography}{99}
\setlength{\itemsep}{0pt}

\bibitem{Ando87}
T. Ando,
On some operator inequalities,
{\em Math. Ann.},
vol. 279, pp. 157--159, 1987.

\bibitem{Baus16}
S. Bartz, H. H. Bauschke, S. M. Moffat, and X. Wang, 
The resolvent average of monotone operators: Dominant and 
recessive properties,
{\em SIAM J. Optim.},
vol. 26, pp. 602--634, 2016.

\bibitem{Livre1} 
H. H. Bauschke and P. L. Combettes, 
{\em Convex Analysis and Monotone Operator Theory in Hilbert 
Spaces,} 2nd ed. 
Springer, New York, 2017.

\bibitem{Baus10}
H. H. Bauschke, S. M. Moffat, and X. Wang,
The resolvent average for positive semidefinite matrices,
{\em Linear Algebra Appl.},
vol. 432, pp. 1757--1771, 2010.

\bibitem{Jota24}
M. N. B\`ui and P. L. Combettes,
Integral resolvent and proximal mixtures,
{\em J. Optim. Theory Appl.},
vol. 203, pp. 2328--2353, 2024.

\bibitem{Svva23}
P. L. Combettes,
Resolvent and proximal compositions,
{\em Set-Valued Var. Anal.},
vol. 31, art. 22, 29 pp., 2023. 

\bibitem{Acnu24}
P. L. Combettes,
The geometry of monotone operator splitting methods,
{\em Acta Numer.},
vol. 33, pp. 487--632, 2024.

\bibitem{Eusi24}
P. L. Combettes and D. J. Cornejo,
Signal recovery with proximal comixtures,
{\em Proc. Europ. Signal Process. Conf.},
pp. 2637--2641, Lyon, France, August 26--30, 2024.

\bibitem{Eect25}
P. L. Combettes and D. J. Cornejo,
Variational analysis of proximal compositions and integral proximal
mixtures.
{\em Evol. Equ. Control Theory},
vol. 17, pp. 106--139, 2026.

\bibitem{Sipr21}
P. L. Combettes and J.-C. Pesquet, 
Fixed point strategies in data science, 
{\em IEEE Trans. Signal Process.}, 
vol. 69, pp. 3878--3905, 2021.

\bibitem{Cora94}
G. Corach, H. Porta, and L. Recht,
Convexity of the geodesic distance on spaces of positive operators,
{\em Illinois J. Math.},
vol. 38, pp. 87--94, 1994.

\bibitem{Svva25}
D. J. Cornejo,
Parametrized families of resolvent compositions,
{\em Set-Valued Var. Anal.},
vol. 33, art. 6, 24 pp., 2025.

\bibitem{Gaub14}
S. Gaubert and Z. Qu,
The contraction rate in Thompson's part metric of order-preserving
flows on a cone -- Application to generalized Riccati equations,
{\em J. Differential Equations},
vol. 256, pp. 2902--2948, 2014.

\bibitem{Hasa20}
M. Hasannasab, J. Hertrich, S. Neumayer, G. Plonka, S. Setzer, and
G. Steidl,
Parseval proximal neural networks,
{\em J. Fourier Anal. Appl.},
vol. 26, art. 59, 31 pp., 2020.

\bibitem{Kami17}
U. S. Kamilov,
A parallel proximal algorithm for anisotropic total variation
minimization,
{\em IEEE Trans. Image Process.},
vol. 26, pp. 539--548, 2017.  

\bibitem{Kim22}
S. Kim,
Mixture and interpolation of the parametrized ordered means,
{\em J. Inequal. Appl.},
vol. 2022, art. 119, 16 pp., 2022.

\bibitem{Kim11}
S. Kim, J. Lawson, and Y. Lim,
The matrix geometric mean of parametrized, weighted arithmetic and
harmonic means,
{\em Linear Algebra Appl.},
vol. 435, pp. 2114--2131, 2011.

\bibitem{Kubo80}
F. Kubo and T. Ando,
Means of positive linear operators,
{\em Math. Ann.},
vol. 246, pp. 205--224, 1980.

\bibitem{Kum15}
S. Kum and Y. Lim,
Nonexpansiveness of the resolvent average,
{\em J. Math. Anal. Appl.},
vol. 432, pp. 918--927, 2015.

\bibitem{Laws01}
J. D. Lawson and Y. Lim, 
The geometric mean, matrices, metrics, and more, 
{\em Amer. Math. Monthly},
vol. 108, pp. 797--812, 2001.

\bibitem{Laws07}
J. Lawson and Y. Lim,
A Birkhoff contraction formula with applications to Riccati
equations,
{\em SIAM J. Control Optim.},
vol. 46, pp. 930--951, 2007.

\bibitem{Laws14}
J. Lawson and Y. Lim,
Karcher means and Karcher equations of positive definite operators,
{\em Trans. Amer. Math. Soc. Ser. B},
vol. 1, pp. 1--22, 2014.

\bibitem{Lee08}
H. Lee and Y. Lim,
Invariant metrics, contractions and nonlinear matrix equations,
{\em Nonlinearity},
vol. 21, pp. 857--878, 2008.

\bibitem{Lemm15}
B. Lemmens and M. Roelands,
Unique geodesics for Thompson's metric, 
{\em Ann. Inst. Fourier (Grenoble)},
vol. 65, pp. 315--348, 2015.

\bibitem{Lim09}
Y. Lim,
Solving the nonlinear matrix equation
$X=Q+\sum_{i=1}^mM_iX^{\delta_i}M_i^*$ via a contraction
principle,
{\em Linear Algebra Appl.},
vol. 430, pp. 1380--1383, 2009.

\bibitem{Lim12}
Y. Lim and M. P\'{a}lfia,
Matrix power means and the Karcher mean,
{\em J. Funct. Anal.},
vol. 262, pp. 1498--1514, 2012.

\bibitem{Mont98}
L. Montrucchio,
Thompson metric, contraction property and differentiability of
policy functions,
{\em J. Econ. Behav. Organ.},
vol. 33, pp. 449--466, 1998.

\bibitem{RedSim}
M. Reed and B. Simon,
{\em Methods of Modern Mathematical Physics, Vol. I: Functional
Analysis},
Academic Press, New York, 1980.

\bibitem{Shen17}
L. Shen, W. Liu, J. Huang, Y.-G. Jiang, and S. Ma,
Adaptive proximal average approximation for composite convex
minimization,
{\em Proc. AAAI Conf. Artif. Intell.},
vol. 31, pp. 2513--2519, 2017.

\bibitem{Thomp63}
A. C. Thompson,
On certain contraction mappings in a partially ordered vector 
space,
{\em Proc. Amer. Math. Soc.},
vol. 14, pp. 438--443, 1963.

\bibitem{Wang11}
X. Wang,
Self-dual regularization of monotone operators via the resolvent
average,
{\em SIAM J. Optim.},
vol. 21, pp. 438--462, 2011.

\bibitem{Yuyl13}
Y.-L. Yu,
Better approximation and faster algorithm using the proximal
average, 
{\em Proc. Conf. Adv. Neural Inform. Process. Syst.}, 
pp. 458--466, 2013.

\end{thebibliography}
\end{document}